\documentclass[12pt]{article}

\usepackage{amsmath,amssymb,graphics, amsthm}
\usepackage[utf8]{inputenc}
\usepackage{graphicx}
\usepackage{latexsym}
\usepackage{graphics}
\usepackage{pdfpages}
\usepackage{graphpap}
\usepackage{pstricks}
\usepackage{pst-node}

\newtheorem{defn}{Definition}

\newtheorem{prop}[defn]{Proposition}

\newtheorem{rmk}[defn]{Remark}

\newtheorem{eg}[defn]{Example}

\makeatother

\numberwithin{equation}{section}

\begin{document}
\title{Translations and extensions of the Nicomachus identity}

\author{Seon-Hong Kim\footnotemark[1], Kenneth B. Stolarsky\footnotemark[2]}

\date{}

\renewcommand{\thefootnote}{\fnsymbol{footnote}}
\footnotetext[1]{Sookmyung
Women's University (shkim17@sookmyung.ac.kr), University of Illinois at Urbana-Champaign (stolarsky.ken@gmail.com)}
\renewcommand{\thefootnote}{\arabic{footnote}}

 \maketitle
\begin{abstract}
    \par\noindent
   We search for Nicomachean identities by adding translation parameters, variable parameters, sequence products and adjoining further numbers to sequences. The solutions of definite and indefinite quadratic forms arise in this study of cubic equations obtained from translation parameters. Our search leads to many general Nicomachean-type identities. We also study the geometry of adjoining two numbers to sequences satisfying the Nicomachean identity.
    \end{abstract}

\par\noindent
\section{INTRODUCTION}\label{intro}
Define an operator $\nu$ on finite sequences of complex numbers $a_i$, i.e., sequences $\sigma=\{a_1, a_2, \cdots, a_n\}$, by
$$
\nu(a_1, a_2, \cdots, a_n)=\left(\sum_{i=1}^n a_i\right)^2-\sum_{i=1}^n a_i^3.
$$
The ancient classical Nicomachean identity is
$$
\nu\left(1, 2, 3, \cdots, n\right)=0.
$$
The above definition of $\nu$ does in fact make sense for the $a_i$ being elements of any ring.
\par
The number and variety of Nicomachean-type identities is very large. Let $\sigma$ be an arbitrary sequence of positive integers. Then Mason has shown \cite{Ma}, by using bag products, that it can be encoded in a Nicomachean-type identity (cube sum = sum squared). In the other direction, Barbeaux and Seraj \cite{Ba} used a simple estimate to show that the number of Nicomachean-type identities for a sequence of length $n$, with $n$ fixed, is finite. For previous studies of various points of views about Nicomachean identity, see \cite{Ba}, \cite{Ce}, \cite{Fe}, \cite{Ki}, \cite{Kl}, \cite{Lo}, \cite{Lu}, \cite{Ma}, \cite{Pa}, \cite{Ra}, \cite{Wa}. Among these the present paper has a point of view most similar to that of Barbeau and Seraj \cite{Ba}, and, like \cite{Ba}, makes use of Pell's equations.
\par
A natural question is ``for what sequences of positive integers or positive rational numbers does $\nu=0$?''. Note that we allow repetitions in our sequences.
\par
We immediately note that aside from the trivial $a_i=0$ for all $i$ case there are three very simple types of sequences with $\nu=0$.
\par
(1) The case $a_i=n$ for $1\leq i\leq n$.
\par
(2) The case in which, for some re-ordering of the $a_i$, each $a_i$ is followed by $-a_i$.
\par
(3) The case $\sigma=\{\sigma_1; \,\sigma_2\}$ where $\sigma_1$ belongs to the case (2) and $\sigma_2$ consists of some integer $m$ repeated $m$ times.
\par\noindent
Here (3) is a combination of (1) and (2). We also note that $\nu(\sigma)$ is the same for all re-orderings of $\sigma$.
\par
Our approach here is to work in a more general setting in which some or all of the $a_i$ are variable parameters, to compute the variable or multivariable function $\nu(\sigma)$, and then to find values of the parameters for which $\nu(\sigma)=0$. Ideally the resulting $a_i$ will be positive rational numbers, but other outcomes may nonetheless be of interest.
\par
For an example, take
$$
\sigma=\{1,2,2,3,4,x,6,8\}.
$$
Then
$$
\nu(\sigma)=(5-x)(x-4)(8+x)
$$
and we have three non-trivial examples of $\nu(\sigma)=0$. Here the pair $(-8,8)$ may be removed.
\par
Another problem is to determine what extensions of a sequence with $\nu=0$ also have this property. For an example, let
$$
\sigma=\{1,2,3,\cdots, 16, a,b\}.
$$
To avoid trivialities assume $a+b\neq 0$. Then $\nu(\sigma)=0$ is equivalent to
$$
a^2-a+b^2-b-ab=272.
$$
A brief search yields the solutions
$$
(a,b)=(2,18), (17, 18), (9,20) \,\,\, \text{and}\,\,\, (12,20).
$$
Each leads to a specific non-trivial solution in positive integers of $\nu(\sigma)=0$, with $(17, 18)$ corresponding to the classical identity. In general, for
$$
\sigma=\{1,2,3,\cdots, n,a,b\}
$$
with $n\geq 5$, $a+b\neq 0$ and $\{a, b\}\neq \{n+1, n+2\}$ the equation $\nu(\sigma)=0$ implies that one of $a$ or $b$ is less than $n-1$. Hence this can produce only sequences $\sigma$ with repetitions. In fact the equation $\nu(\sigma)=0$ can be put into a more attractive form by means of the change of variable
$$
a=u+v+1, \quad b=u-v+1 \,\,\text{or}\,\,b=-u+v+1.
$$
This transforms the equation to
$$
u^2+3v^2=\frac {n^3-1}{n-1}\quad\text{or}\quad 3u^2+v^2=\frac {n^3-1}{n-1},
$$
respectively. For $n\geq 5$ it follows from the above that one of $u$ or $v$ is less than $n-2$, which in turn implies that one of $a$ or $b$ is less than $n-1$.
\par
The above shows that part of our search for Nicomachean identities is tied to the Diophantine analysis of positive definite quadratic forms. We shall see in Section~\ref{trans} that the solutions of indefinite quadratic forms (``Pell's equation'') is also a natural part of this search.
\par
Further pursuit of this ``adjoin $\{a, b\}$ procedure'' leads many general Nicomachean type identities. Of particular interest is
\begin{equation}
\left(\sum_{j=1}^{7n-12}j +(8n-12)+(3n-3)\right)^2=\sum_{j=1}^{7n-12}j^3 +(8n-12)^3+(3n-3)^3.
\label{11id}
\end{equation}
Note the growing gap here between the two largest cubes $(7n-12)^3$ and $(8n-12)^3$. Several natural questions about the gap will be answered in Section~\ref{frac}.
\par
A further main theme is the introduction of translation parameters, especially for sequences $\sigma$ with $\nu(\sigma)=0$. For $\sigma=\{a_1, a_2, \cdots, a_n\}$, let $\sigma+t$ denote the translation of $\sigma$ by $t$, defined by
$$\{a_1+t, a_2+t, \cdots, a_n+t\}.
$$
The simplest example here is
$$
\nu(1+t, 2+t, \cdots, n+t)=-\frac n2 t(t+1)(2t+n+1).
$$
This generalizes the classical identity to which it reduces upon letting $t=0$, the first root of the right-hand side.
\par
Various translations of Nicomachean identities will be considered in this paper. Some translation examples are from the sequences satisfying $\nu=0$ such as
$$
\{1, 1, \cdots, 1, 2,2, \cdots, 2, \cdots, n, n, \cdots, n\},
$$
where each integer $1$ to $n$ is repeated same number of times. Another example is
$$
\{1,2,2,3,4,5,\cdots,m-1, m, m+2\},
$$
where $2$ is added and $m+1$ is deleted in $\{1, 2, \cdots, m+2\}$. Also of interest is
$$
\{a_i b_j : \, 1\leq i \leq p, 1\leq j \leq q\}
$$
where $A=\{a_i\}_{i=1}^p$ and $B=\{b_j\}_{j=1}^q$ satisfy $\nu=0$. This sequential product is called the ``bag product'' in \cite{Ma}. These translations will be studied in Section~\ref{trans} through Section~\ref{thea}.
\par
One could also translate a subsequence $\sigma^*$ of $\sigma$. For example, if $\sigma$ is $\{1,2,3, \cdots, 23\}$ and $\sigma^*$ is the subsequence of elements of $\sigma$ of the form $3k+1$, then $\sigma$ with $\sigma^*$ translated by $t$ is
$$
\sigma(t)=\{1+t, 2,3,4+t,5,6,\cdots, 20, 21, 22+t, 23\}
$$
and
$$
\nu(\sigma(t))=-4t(t+27)(2t-1).
$$
For $t=1/2$ we have a sequence of positive rational numbers, that includes non-integers, for which $\nu(\sigma)=0$. We know of no previous such example. See the front part of Section~\ref{frac} for another example.
\par
In Section~\ref{adjo}, we study the case of adjoining numbers to finite sequences of numbers so that each resulting sequence satisfying $\nu=0$. From the study of  the sequence $\{ n , n , \cdots , n \}$, where the first $n$ appears $n$ times, adjoining an arithmetic progression of length $n$, leads to an interesting sequence with Fibonacci numbers with $\nu(\sigma)=0$. 
\par
In Section~\ref{theg} and Section~\ref{theg1}, we start with a $\sigma$ satisfying
$\nu(\sigma)= 0$, and consider first the case of adjoining one element at a time
to $\sigma$, and then the case of simultaneously adjoining two elements. If $\{a, b\}$ is adjoined to $\sigma$ and $a+b\neq 0$, the $\nu(\sigma; a, b)$ must have the quadratic factor
$$
-a + a^2 - b - a b + b^2 - 2 c,
$$
where $c$ is the sum of elements of $\sigma$. The geometry of such quadratics will be studied in Section~\ref{theg1}.
\vskip0.5cm
\par\noindent
\section{TRANSLATIONS OF NICOMACHEAN IDENTITY AND THEIR REPETITIONS} \label{trans}
We know of no clear reference but it is not hard to show that
\begin{equation}
\nu(1+t, 2+t, \cdots, n+t)=-\frac n2 t(2t+n+1)(t+1).
\label{1tnt}
\end{equation}
Here $t=0$ is the original identity, $t=-(n+1)/2$ leads to an example of the second simple case in the introduction, and $t=-1$ reduces to the original identity with $n$ replaced by $n-1$. The corresponding integral version is
$$
\left(\int_0^x (z+t)\, dz\right)^2-\int_0^x (z+t)^3\, dz=-\frac 12 t^2 x(x+2t).
$$
\par
What can be said if each integer from $1$ to $n$ is repeated $u+1$ times? Let
$$
R=R(u)=R(u)[\,\,\,]
$$
be the operator that replaces any argument $x$ by $u+1$ copies of $x$. For example
$$
R(3)[x]=\{x,x,x,x\}.
$$
Then
\begin{equation}\begin{split}
&\nu \left( \left(R(u)[1], R(u)[2], \cdots, R(u)[n]\right)+t\right)\\
=&-\frac n2 (u+1)(2t+n+1)\left(t(t+1)-u\left( \frac {n(n+1)}2+nt \right)\right).
\nonumber
\end{split}\end{equation}
Here the case $u=0$ reduces to the original identity (\ref{1tnt}), and the case $t=-(n+1)/2$ to the simple case (2) in the introduction as before. The effect of the vanishing of the last factor is less transparent.
\par
The case of $n=2(u+1)$ is of special interest. The last factor becomes
$$
(t+u+1)(t-u(2u+3)).
$$
The case $t=-(u+1)$ causes $\sigma$ to be of the form described in case (3) in the introduction. The case $t=u(2u+3)$ is more interesting. For example, if $u=6$ then $n=14$ and $t=90$. The translation of $\{1,2,3,\cdots,14\}$ by $90$ after a $7$-fold repetition gives us a new sequence for which $\nu=0$.
\par
The search for other cases in which the last factor can be factored leads to Pell's equation. The factorization of the last factor hinges on whether the discriminant with respect to $t$ is a square, i.e., whether
$$
1+n^2(u^2+2u)=y^2
$$
has a solution in integers $n$, $u$ and $y$. This can be written in the form
$$
\frac {y^2-1}{m^2-1}=n^2, \quad m=u+1
$$
or in Pell's equation form
$$
y^2=(m^2-1)n^2+1.
$$
For example, when $m=6$ we have $y^2=35n^2+1$. By standard methods, we quickly find a formula for infinitely many solutions here. For example,
$$
846^2=35\cdot 143^2+1.
$$
Here $u=5$ and $n=143$, so the last factor is
$$
(t-780)(t+66).
$$
The $t=-66$ root tells us that $\{-65, -64, \cdots, -1, 0, 1, \cdots, 76, 77\}$ with each number repeated $6$ times satisfies $\nu=0$, while the meaning of $t=780$ is straightforward. In this $m=6$ case the $y^2=35x^2+1$ solutions have
$$
x=\frac {(6+\sqrt{35})^b-(6-\sqrt {35})^b}{2\sqrt {35}}
$$
and $x=143$ when $b=3$.
\par
In summary, there are an infinity of sequences satisfying $\nu=0$ that are generated by solutions to an infinite number of Pell's equations. We add that the integral analogue here is
\begin{equation}\begin{split}
&\left( \int_0^x (u+1)(z+t)\, dz\right)^2-(u+1)\int_0^x (z+t)^3 \, dz\\
=& -\frac x2 (u+1)(2t+x)\left( t^2-u\left( \frac {x^2}2+xt\right)\right).
\nonumber
\end{split}\end{equation}
\vskip0.5cm
\par\noindent
\section{TRANSLATIONS OF OTHER NICOMACHEAN IDENTITIES} \label{tran3}
There are $\sigma$ with $\nu(\sigma)=0$ that have both arbitrarily small and arbitrarily large translations to other $\nu=0$ sequences. A family of such sequences is given by
$$
\sigma= \{ 1 , 2 , 2 , 3 , 4 , 5 , 7 , 8 , 8 , 10 , 10 , 12 , 12 , \cdots , 2 s , 2 s \},
$$
where the first seven entries are as above and the remaining follow a simple pattern of twice repeated even numbers. It is not hard to show that $\nu =0$ here. Now
$$
 \nu ( \sigma + t ) = - t ( -12 + ( 2 s ^ 2 + 2 s - 1 ) t + ( 2 s + 1 ) t ^ 2 ) .
$$
The above quadratic in $t$ has the root (num)/(den) where
$$
\text{num}=1 - 2 s - 2 s ^ 2 + \sqrt{  49 + 92 s + 8 s ^ 3 + 4 s ^ 4 } \,\,\, \text{and}\,\,\, \text{den}= 2+4s.
$$
Thus $t \rightarrow 0$ as $s \rightarrow \infty$. The product $-12/(1+2s)$ of the quadratic's roots has order $1 / s$, and the other root tends to $-\infty$ as $s \rightarrow \infty$.
\par
The sequence
\begin{equation}
\sigma_m=\{1,2,2,3,4,5,\cdots,m-1, m, m+2\}
\label{hhvv54}
\end{equation}
satisfies $\nu=0$. Then the translation identity
$$
\nu(\sigma_m+t)=t\left(-2(2+t)^2-\frac 12 (t-1)(3+2t)m-\frac 12(t-1)m^2\right)
$$
is a polynomial that is, not surprisingly, cubic in $t$. Here $t=0$ returns us to the identity $\nu(\sigma_m)=0$. We now ask, for what values of $t$ other than $0$ do we have $\nu=0$? The quadratic formula yields the positive root
$$
t=t(m)=\frac {-16-m-m^2+\sqrt{m(-48+73m+10m^2+m^3)}}{4(m+2)}.
$$
This is easily seen to have the limit $1$ as $m \rightarrow \infty$. This suggests that for large $m$,
$$
\nu(\sigma_m+1)=\nu(2,3,3,4,5,\cdots, m-1, m+1)
$$
would be near zero. But the above translation identity shows that
\begin{equation}
\nu (\sigma_m +1)=-18
\label{nuci1}
\end{equation}
for $m\geq 5$. This also follows from the identity
$$
\left( \left(\frac {m(m+1)}2-1\right)+(m+2)+3\right)^2-\left( \frac {m(m+1)}2\right)^2-1+3^3+(m+2)^3=-18.
$$
This is reminiscent of the fact that for any $c\neq 0$ there are sequences of smooth functions $f_n(x)$ on $[0, 1]$ such that $\lim_{n\rightarrow \infty} f_n(x)$ exists, $\int_0^1 f_n(x)dx =0$ for every integer $n\geq 1$, yet $\int_0^1 \lim_{n\rightarrow \infty} f_n(x)dx =c$.
\vskip0.5cm
\par\noindent
\section{LIOUVILLE'S IDENTITY AND SEQUENCE MULTIPLICATION} \label{liou} The earliest non-trivial generalization of the Nicomachean identity is Liouville's
\begin{equation}
\left( \sum_{d\vert n}\tau(d)\right)^2 = \sum_{d\vert n }\tau(d)^3,
\label{lion}
\end{equation}
where $\tau(d)$ is the number of positive divisors of $n$. For $d=p^n$, $p$ a prime, this is simply the original identity. The basic principle here is that if two sequences $A=\{a_i\}_{i=1}^p$ and $B=\{b_j\}_{j=1}^q$ satisfy $\nu=0$, then so does the sequence product (``bag product'')
$$
A^{**}B=\{a_i b_j : \, 1\leq i \leq p, 1\leq j \leq q\}.
$$
We write the product as $A^{**n}$ if it is the product of $n$ copies of the sequence $A$. Then we observe that
\begin{equation}
\nu\left(A^{**n}\right)=\left(\sum_{i=1}^p a_i \right)^{2n}-\left(\sum_{i=1}^p {a_i}^3\right)^n=c \, \nu(A)
\label{nulen}
\end{equation}
for some constant $c$. For example
\begin{equation}
\nu\left(A^{**2}\right)=\left[\left(\sum_{i=1}^p a_i \right)^2+\sum_{i=1}^p {a_i}^3\right] \nu(A).
\nonumber
\end{equation}
We shall examine translations of such products.
\par
Let $n=p_1^{\alpha_1}\cdots p_r^{\alpha_r}$ be the standard prime factorization of $n$. Then the positive divisors of $n$ are of the form $p_1^{a_1}\cdots p_r^{a_r}$, where $0\leq a_i \leq \alpha_i$ for $1\leq i \leq r$, and so (\ref{lion}) implies that
\begin{equation}
\left(\sum_{\substack{0\leq a_i\leq \alpha_i \\1\leq i\leq r}} (a_1+1)(a_2+1)\cdots(a_r+1)\right)^2
=\sum_{\substack{0\leq a_i\leq \alpha_i \\1\leq i\leq r}} \left((a_1+1)(a_2+1)\cdots(a_r+1) \right)^3.
\nonumber
\end{equation}
Add $t$ to each entry of above and consider the squared sum minus the cube sum. Then a straightforward computation yields that
\begin{equation} \begin{split}
&\left(\sum_{\substack{0\leq a_i\leq \alpha_i \\1\leq i\leq r}} \left((a_1+1)(a_2+1)\cdots(a_r+1)+t \right)\right)^2-\sum_{\substack{0\leq a_i\leq \alpha_i \\1\leq i\leq r}} \left((a_1+1)(a_2+1)\cdots(a_r+1)+t \right)^3\\
=&-t \prod_{i=1}^r (\alpha_i+1) f_r(t),
\nonumber
\end{split}\end{equation}
where
\begin{equation} \begin{split}
f_r(t)&=t^2-\left( \prod_{i=1}^r (\alpha_i+1)-3\prod_{i=1}^r \left(\frac {\alpha_i}2+1\right)\right)t-\\
&\quad 2\prod_{i=1}^r (\alpha_i+1)\left(\frac {\alpha_i}2+1\right)+3\prod_{i=1}^r \frac 16 (\alpha_i+2)(2\alpha_i+3).
\label{ttvv45}
\end{split}\end{equation}
We shall consider the roots of the quadratic polynomial $f_r(t)$. This takes us beyond the set of integers. It is easily checked that the discriminant of $f_r(t)$ is positive, and we now create the sequence $t(r,m)$, where $t(r,m)$ is the right-most root of $f_r(t)$. Observe that for a sequence of numbers generated by successive values of a quadratic polynomial, the second difference is constant. The sequence $t(r,m)$ need not be a sequence of values of some polynomial, but it has the property that the second difference converges.
\vskip0.3cm
\begin{prop} For $r=2$ and $\alpha_1=\alpha_2=m-1$, the sequence of second differences of
\begin{equation}
t(2,m)=\frac 1{24}\left(3 (-3 - 6 m + m^2)+\sqrt{ 3(-1 + m)^2 (11 + 34 m + 35 m^2)}\right)
\nonumber
\end{equation}
converges to
\begin{equation}
\frac {3+\sqrt{105}}{12}
\label{conv24}
\end{equation}
and the smaller root converges to the algebraic conjugate of the above as $m \rightarrow \infty$. Hence for $A=\{1,2,3,\cdots,m\}$, the sequence of second differences of the right-most roots of $
\nu\left(A^{**2}+t\right)$ converges to (\ref{conv24}).
\end{prop}
\vskip0.3cm
\par\noindent
\begin{rmk} If $\alpha_1=\alpha_2=\cdots=\alpha_r=k$ for $r\geq 1$, then the $r$th difference of the sequence of the right-most roots seems to converge to a constant as $k\rightarrow \infty$. Those constants are
\begin{equation} \begin{split}
&-\frac 12, \,\, \,\frac{1}{12} \left(3+\sqrt{105}\right),\,\, \,\frac{1}{8} \left(15+\sqrt{545}\right),\,\, \,\frac{1}{12} \left(117+\sqrt{20985}\right),\,\, \,\\
   &\frac{5}{24}\left(261+\sqrt{84761}\right),\,\, \,
\frac{5}{24} \left(1647+\sqrt{3036705}\right)
\nonumber
\end{split}\end{equation}
for $r=1,2,\cdots, 6$, respectively.
\end{rmk}
\vskip0.5cm
\par\noindent
\section{THE $A^{**}B$ CASE, $B$ AN $A$-TRUNCATION} \label{thea}
Consider the case $r=2$ and $\alpha_1=m-1$, $\alpha_2=r-1$ with $r<n$. Then $f_r(t)$ in (\ref{ttvv45}) equals
$$
f_r(t)=\frac 1{12}\left( (m+1) (r+1) (2 m r-2 n-2 r-1)+3 (m r-3 m-3 r-3)t-12t^2\right),
$$
i.e., $\nu(A^{**}B+t)=-mrt f_r(t)$, where $A=\{1,2,3,\cdots,m\}$ and $B=\{1,2,3,\cdots,r\}$. As before, create the sequence of the right-most roots of the above quadratic. Here the first difference of the sequence converges to
\begin{equation}
\frac{1}{24} \left(3 r+\sqrt{3} \sqrt{(5 r+1) (7 r-5)}-9\right).
\label{fds3}
\end{equation}
as $m \rightarrow \infty$. For examples, if $r=2, 3, 4$, (\ref{fds3}) becomes
$$
\frac{1}{8} \left(\sqrt{33}-1\right)=0.593070\cdots, \,\,\, \frac{2}{\sqrt{3}}=1.15470\cdots, \,\,\, \frac{1}{8} \left(1+\sqrt{161}\right)=1.71107\cdots.
$$
These have minimal polynomials
$$
-2 + x + 4 x^2, \,\,\, -4 + 3 x^2, \,\,\, -10 - x + 4 x^2,
$$
respectively. We now rescale the minimal polynomials of the various limiting values of the first differences. For example, replace $-4+3x^2$ and $-10-x+4x^2$ by
$$
-8+6x^2 \quad \text{and}\quad -20-2x+8x^2.
$$
By doing this to get successive even coefficients for the $x^2$ term, we find a general formula:
\begin{equation}
-\frac {m(m+1)(m+2)}3-\frac 12 (m-2)(m+1)x+2(m+1)x^2.
\label{761w}
\end{equation}
This (\ref{761w}) has rather nice discriminants:
\begin{equation}\begin{split}
\Delta_x &= \frac 1{12} (m+1)^2(6+5m)(2+7m),\\
\Delta_m &= \frac 1{6^4} (2+9x+12x^2)^2(16+72x+105x^2).
\nonumber\end{split}\end{equation}
Here the first difference of the sequence of the right-most roots of (\ref{761w}) converges to
$$
0.551956\cdots=\frac {3+\sqrt{105}}{24}
$$
that is same as (\ref{conv24}) divided by $2$. This seems to not be a coincidence.
\vskip0.5cm
\par\noindent
\section{FRACTIONS, LARGE DENOMINATORS, AND GAPS} \label{frac}
By translation of subsets we can find additional solutions to $\nu(\sigma)=0$ with the elements of $\sigma$ being positive rational numbers not all integers. Here are two examples involving arithmetic progressions modulo $7$. For $\sigma^*$ a subset of $\sigma$, let $\sigma / \sigma^*$ be the set obtained from $\sigma$ by removing all elements of $\sigma^*$. Write
$$
\{\sigma; \sigma^* \rightarrow \sigma^*+t\}=(\sigma / \sigma^*)\cup (\sigma^*+t).
$$
If $\sigma=\{1,2,\cdots, 67\}$ and $\sigma^*=\{3, 10, 17, \cdots, 66\}$ we have
$$
\nu(\sigma; \sigma^* \rightarrow \sigma^*+t)=-5t(t+91)(2t+5)
$$
and $t= -5/2$ gives an example. For $\sigma=\{1,2,\cdots, 19\}$ and $\sigma^*=\{6, 13\}$, we have a shorter example. Here
$$
\nu(\sigma; \sigma^* \rightarrow \sigma^*+t)=-t(t+29)(2t-5)
$$
and $t= 5/2$ gives
\begin{equation}
\tau=\left\{1,2,3,4,5,7,8,\frac {17}2, 9,10,11,12,14,15, \frac {31}2, 16,17,18,19\right\}
\label{ta999}
\end{equation}
with $\nu(\tau)=0$.
\vskip0.3cm
\begin{prop} There exist a sequence $\sigma$ with positive rational elements and non-integers $p/q$, with $(p, q)=1$ and $q$ arbitrarily large, such that $\nu(\sigma)=0$.
\end{prop}
\begin{proof} Iterate $w\rightarrow \tau^{**}w$ with initial condition $w=\tau$ in (\ref{ta999}). Then the $n$th iteration has a positive rational element $p/2^{n+1}$ where $p$ is odd, e.g., $p=17^{n+1}$. Next
$$
\nu(\tau^{**(n+1)})=195^2 \nu(\tau^{**n}),
$$
where $195$ is the sum of all elements of the sequence $\tau$. In fact, the sum of all elements of $\tau^{**(n+1)}$ is the sum of all elements of $\tau^{**n}$ multiplied by $195$ and the sum of cubes of all elements of $\tau^{**(n+1)}$ is the sum of cubes of all elements of $\tau^{**n}$ by $38025=195^2$, where $38025$ is the sum of cubes of all elements of $\tau$. Since $\nu(\tau)=0$, the result follows.
\end{proof}
\vskip0.3cm
\par
How large is the gap of the elements of the sequence with positive integer entries such that $\nu=0$?
By iterating $w\rightarrow \{1,2\}^{**}w$ with initial condition $w=\{1, 2\}$, we produce
sequences having the largest gaps that are halves of their lengths. The proof of Proposition~\ref{ppals0} demonstrates this. There is a sequence that have the largest gap exceeding half of its length. An example is
$$
\{ 6 , 6 , 7 , 7 , 8 , 8 , 8 , 9 , 9 , 10 , 10 , 17 \}.
$$
So if they exist, there seems to be three types of sequences with positive integer entries and $\nu=0$ concerning the largest gaps among their elements:
\par\noindent
(1) for any $\epsilon>0$, the sequences where the ratio of the largest gap to the length is greater than $1 - \epsilon$,
\par\noindent
(2) the sequences of independent interest that have size-able the largest gaps,
\par\noindent
(3) the sequences having the property that the gap between the two largest integers can be arbitrarily large.
\par
There exist sequences that satisfy the case (1) mentioned above.
\vskip0.3cm
\par\noindent
\begin{prop} For any $\epsilon>0$, there is a sequences with only two positive integer items such that
$$
\frac {\text{the largest gap}}{\text{length}}>  1 - \epsilon.
$$
\end{prop}
\begin{proof} Consider the sequence $\{ n , n , n , \cdots, n \}$ with length $n$. Also consider
$$
\{ n ^ 2 , n ^ 2 , \cdots, n ^ 2 \}
$$
with length $n^2 - 2$. Joining these sequences results in a sequence with length $n^2 + n - 2$, and the largest gap in this sequence is $n ^ 2 - n$. Therefore, the ratio of the largest gap to the length is given by
$$
\frac {n ^ 2 - n}{ n ^ 2 + n - 2}
$$
which tends to $1$ as $n \rightarrow \infty$. However $\nu = 0$ for this sequence. In fact, the sum of the elements in the sequence is given by
$$
n ^ 2 + ( n ^ 2 - 2 ) n ^ 2 =  n ^ 4 - n ^ 2 ,
$$
while the sum of the cubes of the elements is $n ^ 4 + ( n ^ 2 - 2 ) n ^ 6 = n ^ 8 - 2 n ^ 6 + n ^ 4$. Since the latter is the square of the former, the result follows.
\end{proof}
\vskip0.3cm
\par
We provide an example related to the sequences mentioned in case (2). Consider the sequence:
$$
 \{ 1 , 2 , 3 , \cdots, 94,  94+1, m, m, \cdots, m \},
 $$
 where $m$ is repeated $2\cdot 94$ times, and we choose $m=228$. This sequence satisfies $\nu=0$, and the largest gap in this sequence is $228-95=133$. The general pattern for such sequences is given by:
\begin{equation}
 \{ 1 , 2 , 3 , \cdots, j,  j+1, m, m, \cdots, m \},
 \label{mjsq000}
\end{equation}
 where positive integer $m$ is repeated $2j$ times, for appropriate $j$ and $m$. This implies that
\begin{equation}
m=j + \sqrt{2j^2 + 3 j + 2},
\nonumber
\end{equation}
and it is related to the Pell-type equation
$$
              2 x ^ 2 + 3 x + 2 = r ^ 2.
$$
It is interesting that a positive integer $r$ satisfying the above equation is connected to the coefficients of certain Chebyshev polynomials. References for this can be found in \cite{Ra} and in the OEIS integer sequences A056161 and A055979. By elementary methods, we can obtain the following results. We omit the proof for brevity.
\vskip0.3cm
\par\noindent
\begin{prop}  If the sequence (\ref{mjsq000}) satisfies $\nu=0$, then the $j$'s are generated by
$$
\frac {x ^ 4 + x ^ 3 - 19 x ^ 2 - 5 x - 2}{( x - 1 ) ( 1 - 6 x + x ^ 2 ) ( 1 + 6 x + x ^ 2 )},
$$
and the corresponding $m$'s are generated by
$$
\frac {6 ( 1 + x ) ^ 2}{( 1-x ) ( 1 - 6 x + x ^ 2 ) ( 1 + 6 x + x ^ 2 )}.
$$
\end{prop}
\vskip0.3cm
\par\noindent
Finally, we give a constructive proof about the existence of the third type of sequences.
\vskip0.3cm
\begin{prop} There are sequences with positive integer entries in which the gap between the largest two integers can be arbitrarily large and $\nu = 0$.
\label{ppals0}
\end{prop}
\begin{proof} Iterate $w\rightarrow \{1,2\}^{**}w$ with initial condition $w=\{1, 2\}$. Then we see that $\{1,2\}^{**(n+1)}$ is the join of $\{1,2\}^{**n}$ and $2\{1,2\}^{**n}$. So the sequence $\{1,2\}^{**(n+1)}$ ends in $2^n$ and $2^{n+1}$ in increasing order and has length $2^{n+1}$ while $2^{n+1}-2^n=2^n$. Now $\nu(1,2)=0$ and
$$
\nu\left(\{1,2\}^{**(n+1)}\right)=3^2 \nu \left(\{1,2\}^{**n}\right)
$$
since the sum of all elements of $\{1,2\}^{**(n+1)}$ is the sum of all elements of $\{1,2\}^{**n}$ multiplied by $3$ and
the sum of cubes of all elements of $\{1,2\}^{**(n+1)}$ is the sum of cubes of all elements of $\{1,2\}^{**n}$ by $3^2(=1^3+2^3)$. Since $\nu\left(\{1,2\}^{**n}\right)=0$, the result follows.
\end{proof}
\vskip0.3cm
\par
Finally we introduce a sequence containing some interesting irrational numbers with $\nu=0$.
\vskip0.3cm
\par\noindent
\begin{prop} There exist a sequence $\sigma$ with positive algebraic integers from $\Bbb Q(\sqrt 5)$ such that more than half of $\sigma$ are irrational and $\nu(\sigma)=0$. Let $\phi$ be the golden ratio. Then given $N > 0$, there exists such a sequence
containing all of the numbers $(1/\phi)^n$ for $1\leq n \leq N$.
\end{prop}
\begin{proof} Take  a set $\{ 1 , 2 , 3, 4,5, 6 \}$ and multiply (rather than translate by $t$) the subset $\{2, 4, 6\}$ by the variable $x$. Then for $A_x=\{1 , 2x , 3, 4x,5, 6x\}$,
$$
\nu \left(A_x\right) = -72 (x-1) (4 x^2+2x-1)=0
$$
when $x = ( 1 / 4 ) \left( - 1 + \sqrt 5\right)$, and $2 x = 1 / \phi$. This $x$ gives us positive algebraic integers $2x$, $4x$, $6x$ with minimal polynomials
$$
      x ^ 2 + x - 1 , \,\,\, x ^ 2 + 2 x - 4 , \,\,\, x ^ 2 + 3 x - 9,
$$
respectively. Iterate $w\rightarrow A_x^{**}w$ with initial condition $A_x$, where $x = ( 1 / 4 ) \left( - 1 + \sqrt 5\right)$ a sufficient number of times. Then $\nu=0$ for the resulting sequence by (\ref{nulen}), and since the product of a non-zero rational with a positive irrational is irrational, the irrationals will ultimately dominate. Also, since $A_x$ has both $1$ and $1 / \phi$, we obtain the powers of $1 / \phi$.
\end{proof}

\vskip0.5cm
\par\noindent
\section{ADJOINING NUMBERS TO SEQUENCES} \label{adjo}
For $m\geq 2$,
$$
\nu(\sigma)=-(x-m-1)x (x+m),
$$
where $\sigma=\{1,\,2,\, 3, \,\cdots\,,\,m,\,x\}$. This says if we want $x$ to preserve the Nicomachian relation we can go forward with $x = m+1$, go back with $x = -m$ , or do nothing $( x = 0 )$. Going forward involves the number $m+1$, the ``positive root''.
\par
Toward ``positive root'' direction, we let $\{a_1, \,a_2, \,\cdots\,, \,a_n\}$ be a finite sequence of real numbers, and keep
adjoining nonzero real numbers $x_1, \,x_2, \,\cdots$ with $x_j\neq -x_{j-1}$ for each $j$ so that each resulting sequence satisfies $\nu=0$. Then for $l\geq 2$, it follows from
$$
\nu\left(a_1, a_2, \cdots, a_n, x_1, x_2, \cdots, x_l\right)=\nu\left(a_1, a_2, \cdots, a_n, x_1, x_2, \cdots, x_l, x_{l+1}\right)=0
$$
that
\begin{equation}\begin{split}
&\left(\sum_{k=1}^n a_k+x_1+x_2+\cdots+x_l\right)^2+x_{l+1}^3
=\sum_{k=1}^n a_k^3+x_1^3+x_2^3+\cdots+x_l^3+x_{l+1}^3\\
=&\left(\left[\sum_{k=1}^n a_k+x_1+x_2+\cdots+x_l\right]+x_{l+1}\right)^2
\nonumber
\end{split}\end{equation}
and so
\begin{equation}
x_{l+1}^2=x_{l+1}+2\left(\sum_{k=1}^n a_k +x_1+\cdots+x_l\right).
\label{bbnn10}
\end{equation}
Similarly we get
\begin{equation}
x_{l}^2=x_{l}+2\left(\sum_{k=1}^n a_k +x_1+\cdots+x_{l-1}\right).
\label{bbnn11}
\end{equation}
Substraction (\ref{bbnn11}) from (\ref{bbnn10}) gives
$$
x_{l+1}^2-x_l^2=x_{l+1}-x_l+2x_l=x_{l+1}+x_l,
$$
which implies that the adjoined numbers
$$x_2, \,x_3,\, x_4, \, \cdots
$$
form an arithmetic progression with difference $1$. Putting the above together, we have the following.
\vskip0.3cm
\par\noindent
\begin{prop} Let
$$
\{a_1, \,a_2, \,\cdots\,, \,a_n\}
$$
be a finite sequence of real numbers, and keep
adjoining nonzero real numbers
$$
x_1, \,x_2, x_3,\,\cdots
$$
with $x_j\neq -x_{j-1}$ for each $j$ so that each resulting sequence satisfies $\nu=0$. Then the adjoined numbers
$$x_2, \,x_3,\, x_4, \, \cdots
$$
form an arithmetic progression with difference $1$.
\label{pf7776}
\end{prop}
\vskip 0.3cm
\par\noindent
As an example, consider
$$
\sigma=\{a_1,\, a_2, \,\cdots\,, \,a_{m+1}\}=\{ 1, \, h,\, h, \,\cdots, \,h\},
$$
where  a real number $h$ is repeated $m$ times, and suppose
$$
\{ 1, \, h,\, h, \,\cdots, \,h, \,x_1, \,x_2, \,\cdots\},
$$
where $x_j\neq -x_{j-1}$ for each $j$, and each sequence
$$
\{ 1, \, h,\, h, \,\cdots, \,h, \,x_1, \,x_2, \,\cdots, \, x_j\},\,\,\, (j\geq 2),
$$
satisfies $\nu=0$. Then by Proposition~\ref{pf7776}, $\{x_2,\, x_3, \,x_4, \,\cdots\}$ is an arithmetic progression with difference $1$. Then using computer algebra, we can show that the substraction of $1/2$ from the first element of this arithmetic progression, i.e. $x_2-1/2$, is a root of the polynomial
$$
P(h,m,x)=c_0+c_2x^2+c_4x^4+x^6,
$$
where
\begin{equation}\begin{split}
c_0&=\frac{1}{64} \left(-512 h^3 m^3-1728 h^2 m^2+8 h \left(64 h^2-243\right) m-225\right),\\
c_2&=\frac{1}{16} \left(192 h^2 m^2+432 h m+259\right),\\
c_4&=\frac{1}{4} (-24 h m-35).
\nonumber
 \end{split}\end{equation}
 Moreover
 $$
 P(h,m,x)-P(m,h,x)=8 h m (h-m) (h+m).
 $$
 \par
Next we obtain an interesting sequence with Fibonacci numbers satisfying $\nu=0$. The sequence
\begin{equation}
\{ n , n , \cdots , n \},
\label{ncopis}
\end{equation}
where the first $n$ copies $n$ times, has $\nu=0$. A natural problem might be ``when can one extend the sequence (\ref{ncopis}) by an arithmetic progression of length $n$ to obtain a sequence with $\nu=0$?''. Let
$$
S=\{ n , n , \cdots , n , a , a+s , a+2s , \cdots , a+(n-1)s \},
$$
where the first $n$ copies $n$ times. This is made of a join of (\ref{ncopis}) and an arithmetic progression of length $n$ with difference $s$. Then a direct computation yields
$$
\nu(S)= \left( n a +  \frac 12 n ( n - 1 ) s \right)  \left( 2 n ^ 2 + n a - a ^ 2 + \frac 12  n ( n - 1 )  ( s - s ^ 2 ) - ( n - 1 ) a s \right).
$$
If the first factor is set equal to zero, then
$$
a=-\frac{1}{2} (n-1) s
$$
and $S$ becomes the join of $\{ n , n , \cdots , n\}$ and
$$
\left\{ -\frac 12 (n-1 ) s,  -\frac 12 (n-3) s,  -\frac 12 (n-5) s, \cdots,  \frac 12 (n-5 ) s, \frac 12 (n-3) s, \frac 12  (n-1) s  \right\}
$$
whose sum and the sum of cubes are both zero, and so only trivialities emerge. So our main question would be ``for what $n$ can the second factor be divisible by $a - h$ where $h$ is a positive integer?''.
\par
The second factor has the discriminant
$$
\text{Disc}_a= ( 9 - s ^ 2 ) n ^ 2 + s ^ 2.
$$
This is negative for $n\geq 2$ and $s\geq 4$. For $s=3$, the second factor becomes
$$
-(-3 + a + n) (a + n).
$$
Hence only $s\leq 2$ is of interest. Below $F_k$ represents a Fibonacci number.
\vskip 0.3 cm     \par\noindent
\begin{prop} For $s=1$, the second factor is divisible by $a - h$ where $h$ is a positive integer if and only if $n^2$ is a triangular number. For $s = 2$, the second factor is divisible by $a - h$ where $h$ is a positive integer if and only if $n = F_{2k}$ where $k$ is a positive integer. Moreover in this case
$$
a=F_{2k-1}+1.
$$
\end{prop}
\begin{proof} For $s=1$, the second factor is
\begin{equation}
2n^2+a-a^2
\label{12max}
\end{equation}
and its discriminant is
$$
\text{Disc}_a = 1+8n^2.
$$
This is a perfect square exactly when $n^2$ is a triangular number, and in this case $(\ref{12max})=0$ implies that
$$
a=\frac 12 (1+\sqrt{1+8n^2})
$$
 is a positive integer. For $s = 2$ the second factor above is
$$
G = n ^ 2 + ( 1 - a ) n - a ^ 2 + 2 a
$$
and
$$
\text{Disc}_a G=4 + 5 n ^ 2 .
$$
This is a perfect square exactly when $n= F_{2t}$, where $t\geq 1$. In this case
\begin{equation}\begin{split}
a=&\frac{1}{2} \left(2-n+\sqrt{5 n^2+4}\right)=\frac{1}{2} \left(2-F_{2t}+\sqrt{5 F_{2t}^2+4}\right)\\
=&F_{2t-1}+1.
\label{ohjn}
\end{split}\end{equation}
\end{proof}
\vskip 0.3cm
\par\noindent
From above proposition, we get interesting sequences with Fibonacci numbers satisfying $\nu=0$.
\vskip 0.3cm
\par\noindent
\begin{eg} Let
$$
S = \{ F_{2 n - 1 } + 1 , F_{2 n - 1 } + 3 , F_{2 n - 1 } + 5 , \cdots , F_{2 n + 2} - 1 \} .
$$ Then $S$ is an arithmetic progression with difference $2$ and length $F_{2 n}$. Let $A = F_{2n}^{ ( F_{2 n} )}$. Then the following sequences have $\nu=0$:
\begin{equation}\begin{split}
&A,\\
&\{ A; \,\, S \},\\
&\{ A; \,\,S; \,\,S + F_{2n}\},\\
&\{ A ; \,\,S ; \,\,S + F_{2n}; \,\,S + 2 F_{2n}\} , \\
&\qquad \vdots\\
&\{ A ; \,\,S ; \,\,S + F_{2n}; \,\,S + 2 F_{2n};\,\,\cdots\,\, ; \,\,S+m F_{2n}\}.
\nonumber
\end{split}\end{equation}
\end{eg}

 \vskip 0.5cm
 \section{ADJOINING A FEW NUMBERS TO SEQUENCES} \label{theg}
 There is a fairly simple type of sequence of positive integers with
the following property: if a single variable is adjoined and the $\nu$ operator applied, the
resulting cubic factors completely over the integers, and two of the roots
are consecutive positive integers.
\vskip0.3cm
\par\noindent
\begin{eg} An example is
$$
\{ 1 , 2 , 2 , 3 , 4 , 4 , 5 , 6 , 8 , 10 ,12 , x \},
$$
where $\nu$ of this equals $- ( x - 6 ) ( x - 7 ) ( x + 12 )$. Another interesting example is the join, say $A_n$, of
$$
\{ 1 , 2 , 3 , \cdots , 2 n + 1 \}, \quad \{ 2 , 4 , 6 , \cdots , 2 n \}, \quad \{ 2 ( n + 1 ) , 2 ( n + 2 ) , \cdots , 2 ( n + ( n + 2 ) ) \}.
$$
Here
$$
 \nu(A_n; x)=- (2 + 2 n - x) (3 + 2 n - x) (4 + 4 n + x).
$$
\end{eg}
\vskip0.3cm
\par
 In another direction, one could search for $\sigma$ with $\nu(\sigma)=0$ for many $\sigma$ of the form
$$
\{1, 2, 3, \cdots, m, a, b\}
$$
and look for patterns, especially for $m$ belonging to particular arithmetic progressions. An infinite number of such identities with $m$ of the form $7n-12$ is given by the identity (\ref{11id}) of the introduction. It has the remarkable feature that if the $3n-3$ is replaced by $5n-8$ on each side in (\ref{11id}), the equality remains valid. There is also a ``nearby'' identity, namely
\begin{equation}\begin{split}
&\left(\frac {(7n-10)(7n-9)}2+(8n-10)+(3n-4)\right)^2\\
=&\left(\frac {(7n-10)(7n-9)}2\right)^2+(8n-10)^3+(3n-4)^3.
\nonumber
\end{split}\end{equation}
Here, if the $3n-4$ terms are replaced by $5n-5$, the equality remains true. Another identity of this type is
\begin{equation}\begin{split}
&\left(\frac {(13n+35)(13n+36)}2+(15n+42)+(7n+21)\right)^2\\
=&\left(\frac {(13n+35)(13n+36)}2\right)^2+(15n+42)^3+(7n+21)^3,
\nonumber
\end{split}\end{equation}
and this remains true with $(7n+21)$ replaced by $(8n+22)$. Yet another is
\begin{equation}\begin{split}
&\left(\frac {(13n+42)(13n+43)}2+(15n+50)+(7n+3)\right)^2\\
=&\left(\frac {(13n+42)(13n+43)}2\right)^2+(15n+50)^3+(7n+3)^3,
\nonumber
\end{split}\end{equation}
which remains true if $7n+3$ is replaced by $(8n+28)$.
\par
There is another observation to make here. In the identity (\ref{11id}), replace every $7n$ by the variable expression $tn$. Then the difference of the two sides is
$$
n(t-7)(11n-15)(n(t+7)-23).
$$
In fact each side satisfies, as a function of $t$, the functional equation
$$
f(7)=f\left( \frac {23}n-7\right).
$$
Similar remarks to apply to the other identities.
\par
For an example in which the two sides have the same general form as those displayed here, but whose difference is a non-zero constant, see (\ref{nuci1}).
\par
Next, here is an ``adjoin $\{a, b\}$ procedure'' that starts with a $\sigma$ such that  $\nu(\sigma)$ is not zero. It leads to a new infinite class of sequences satisfying the Nicomachean identity. Omit $m$ from $\{ 1 , 2 , 3 , \cdots, n \}$ and then adjoin $\{a , b \}$. For the case of $m=1$, there are two non-integer rational numbers $a$ and $b$ extending $\{ 2 , 3 , 4 , \cdots , n \}$ that satisfy Nicomachean identity. To show this we let
$$
p(n) =(1 + n) (1 + 2 n + 5 n^2 + n^3)
, \quad q ( n ) =(1 + 2 n) (1 + n + n^2) .
$$
Note that
$$
p ( n ) - p ( - ( n + 1 ) ) = ( 1+2 n  ) ^ 3 .
$$
Define a sequence $R(n)$ of $n + 2$ elements by
$$
R ( n ) = \left \{ 2 , 3 , 4 , \cdots , n , \frac { - p ( - ( 1+n  ) )}{q(n)},  \frac { p (n )}{q(n)} \right\}.
$$
Then $R = R ( n )$ satisfies Nicomachen identity. Note that the sum of the additional two elements is
$$
\frac { - p ( - ( 1+n ) )}{q(n)}+  \frac { p (n )}{q(n)}   =\frac { ( 1+2 n ) ^ 2}{ 1 + n + n ^ 2 } .
$$
For $m\geq 2$, we have $\{1,2,\cdots, m-1,m+1,m+2, \cdots, n, a, b\}$. Take
$$
(2n+1) (1 -  3m + 3 m ^ 2 +n+n ^2 )
$$
for the denominator of both $a$ and $b$. Then use $( n - ( m - 1 ) )  A$ as the numerator of $a$ and $(n + m ) B$ of $b$, where
$$
 A = 3 m^ 3+ ( - 1 + 3 m + 3 m ^ 2 ) n +  ( 3 m - 1 ) n ^ 2 - n ^ 3
$$
and
$$
B = 1 - 3 m ^ 2 + 3 m ^ 3 + ( 2 + 3 m -3 m ^ 2 ) n + ( 3 m + 2 ) n ^ 2 + n ^ 3  .
$$
Then adjoin $\{ a , b \}$ to $\{ 1 , 2 , 3 , \cdots , n \}$  where the $m$($\geq 2$) has been removed. This sequence satisfies the Nicomachean identity.  A useful check here is that
\begin{equation}
a + b = \frac { m( 2 n + 1 ) ^ 2 }{ 3 m ^ 2 - 3 m + 1 + n + n ^ 2 } .
\nonumber
\end{equation}
\vskip 0.5cm
\par\noindent
\section{THE GEOMETRY OF ADJOINING $\{a, b\}$ TO SEQUENCES} \label{theg1}
\begin{figure}[h]
    \centering
    \includegraphics[scale=0.6]{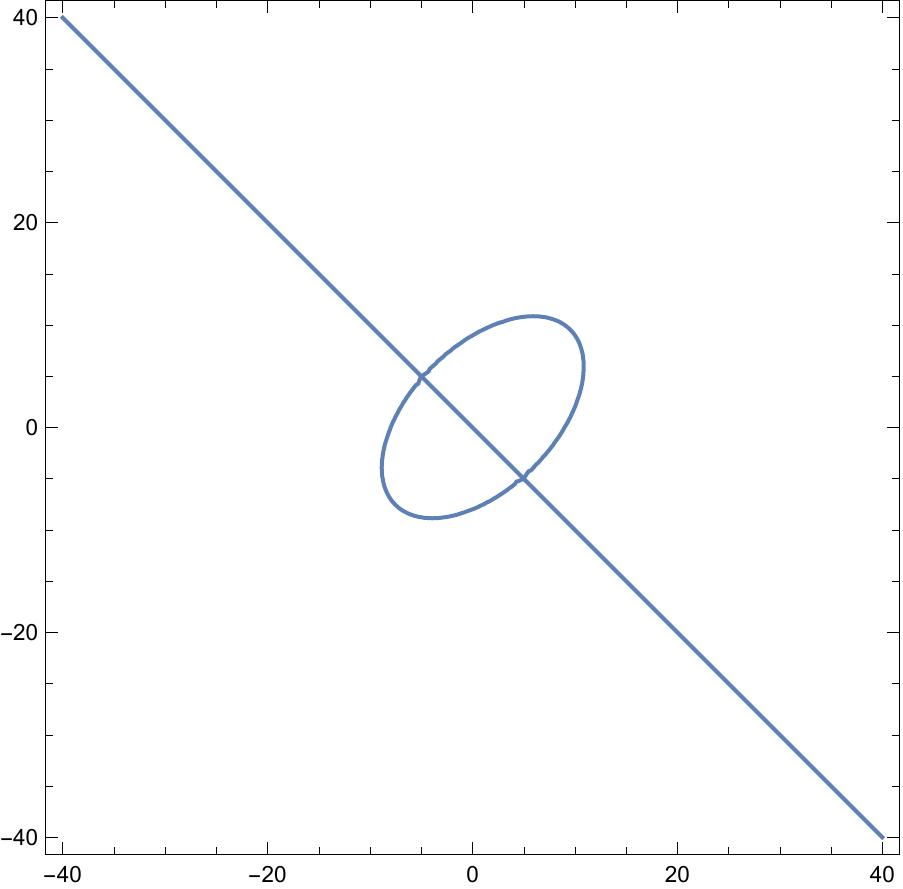}
    \caption{$\nu(\{1,2,\cdots, n, a,b\})=0$, where $n=8$}
\end{figure}
\par\noindent
 It seems peculiar that definite and indefinite quadratic forms play a significant role in a study of a cubic nature. But when $\nu(\sigma)=0$ and $\{a, b\}$ is adjoined to $\sigma$, the $\nu(\sigma; a, b)$ must have a factor $a+b$. In fact,
\begin{equation}\begin{split}
\nu(\sigma; a, b)&=\nu(\sigma)+a^2+b^2+2c(a+b)+2ab-(a^3+b^3)\\
&=-(a + b)(-a + a^2 - b - a b + b^2 - 2 c),
\nonumber
\end{split}\end{equation}
where $c$ is the sum of elements of $\sigma$. Upon discarding this factor we are left with quadratic polynomials. The ellipse formed by setting any these quadratic polynomials equal to zero has eccentricity $\sqrt {2/3}$.
\par
These quadratic polynomials depend on $c$. For example, the two sequences
$$
\sigma_1= \{ 1 , 2 , 3 , 3 , 3 , 5 , 5 , 6 , 7 , 9 , 10 , 12 \}, \quad
\sigma_2 =\{ 1 , 2 , 3 , \cdots , 10 , 11 \}
$$
satisfy $\nu=0$ and $c=66$ is the sum of elements of both sequences, and the corresponding ``adjoined'' quadratic polynomial is $a - a^2 + b + a b - b^2 +132$. The three sequence products originated from $\sigma_1$ and $\sigma_2$ are
$$
\sigma_1^{**} {\sigma_1}, \,\,\, \sigma_1^{**} \sigma_2 , \,\,\, \text{and}\,\,\,  \sigma_2^{**} \sigma_2
$$
that induce the same ``adjoined'' quadratic polynomial with $c=66^2$. Here the ``three'' is from $\binom {2+2-1}{2}=3$, the number of combinations with repetition. Continuing to triple products, we have $\binom {2+3-1}{3}=4$
$$
(\sigma_1^{**} {\sigma_1})^{**} {\sigma_1},  \,\,\,( \sigma_1^{**} \sigma_1)^{**} {\sigma_2}, \,\,\, \,\,\,( \sigma_1^{**} \sigma_2)^{**} {\sigma_2},\,\,\,( \sigma_2^{**} \sigma_2)^{**} {\sigma_2},
$$
that induce the same quadratic polynomial with $c=66^3$. In general we have $\binom {2+N-1}{N}=N+1$ sequence products with the same quadratic polynomial, where $c=66^{N}$ and we can obtain the following.
\vskip0.3cm
\par\noindent
\begin{prop} For any positive integer $N$ there are at least $N$ distinct sequences $\sigma$ with $\nu(\sigma)=0$ such that each $\nu(\sigma; a, b)$ has the same quadratic polynomial factor.
\end{prop}
\vskip 0.3cm
\par\noindent
 If we retain the $a+b$ factor and plot $\nu(\sigma; a, b)=0$ in the $(a, b)$ plane, we have a line cutting through an ellipse. Since a line is a degree $1$ curve and an ellipse a degree $2$ curve, we have $(1+2=3)$ a degenerate cubic. For example, Figure 1 shows the locations of $(a, b)$ where $\nu(\{1,2,\cdots, 8, a,b\})=0$.
\par
Now
$$
w_n=\{1,2,2,3,4,5,\cdots,n\}, \,\,\, n\geq 3,
$$
has
$$
\nu(w_n)=2 (-2 + n + n^2)\neq 0,
$$
and the locus of $\nu(w_n; a, b)=0$ is such a typical cubic curve. Here we have
\begin{equation}\begin{split}
\nu(w_n; a, b)&=\frac 14 (-32 - n^2 - 2 n^3 - n^4) - a^3 - b^3 + \left(\frac 12 (4 + n + n^2) + a +
    b\right)^2\\
&=2 (-1 + n) (2 + n)+ (4 + n + n^2)a+a^2 - a^3+ (4 + n + n^2)b+2 a b+b^2 - b^3
\nonumber
\end{split}\end{equation}
A natural way to search for $a$ and $b$ with $\nu(w_n; a, b)=0$ is to put the cubic into Weierstrass normal form. Consider
$$
h_n ( u , y) = -3 y ^ 2  + \frac 1{16}  (2 u + n + 2 )  P_n(u),
$$
where
$$
P_n ( u ) =3 (n + 2)^3 + 2 (n^2 + 20) u - 4 (3 n + 2) u^2 - 8 u^3
$$
is a cubic polynomial in $u$. Then $h_n ( u , y)=0$ is essentially in Weierstrass normal form. With
$$
u=\frac {a+b}2-\frac {n+2}2, \quad y=\frac {a^2-b^2}4,
$$
we may compute that the $h_n ( u , y)$ becomes
$$
\frac 14 (a + b ) \nu(w_n; a, b).
$$
\par
In the above particular example with $n=8$, the curve $\nu(w_8; a, b)=0$ is given by
$$
140 +76a+a^2-a^3 +76b+ b^2-b^3+2ab =0.
$$
As shown in Figure 2, the locus of this consists of a smooth finite oval disjoint from a smooth infinite curve that is a configuration reminiscent of a textbook plot of certain typical real loci of cubic (elliptic) curves. Moreover, as $n$ increases the corresponding sequence of cubic curves displays a nested sequence of smooth ovals. For this see Figure 3. For the case $n=8$ again, a numerical search gives the points $(0,10)$, $(10,0)$, $(9, 11)$ and $(11,9)$ on the infinite curve and many integer points with coordinates at most $0$ on the finite oval. For example, $(0,-2)$, $(0-7)$, $(-2,-8)$ and $(-8, -2)$ are there. For $(a, b) = (11, 9)$ we have
$$
\tau = \{1, 2, 2, 3, 4, \cdots, 9, 11\}.
$$
an example of one of the $\sigma_m$ sequences studied in Section 3.
\begin {figure}[h]
    \centering
    \begin{minipage}[b]{0.4\textwidth}
    \includegraphics[width=\textwidth] {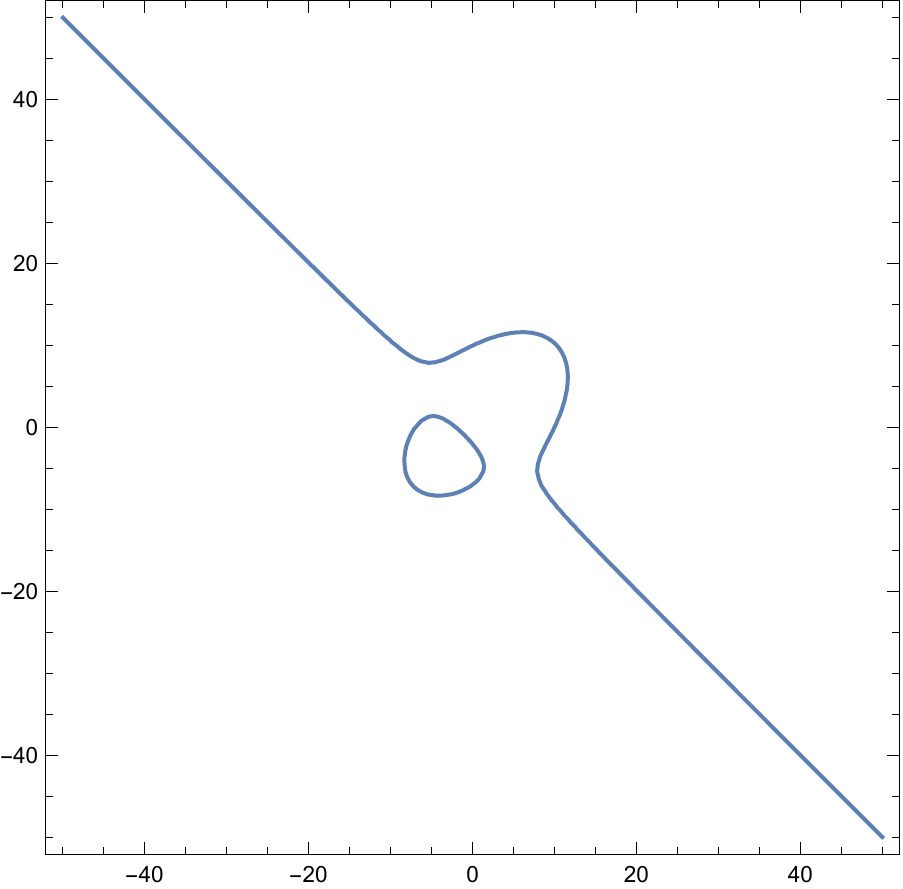}
    \caption {$\nu(w_8; a, b)=0$}
\end {minipage}
 \hfill
 \begin{minipage}[b]{0.4\textwidth}
    \includegraphics[width=\textwidth] {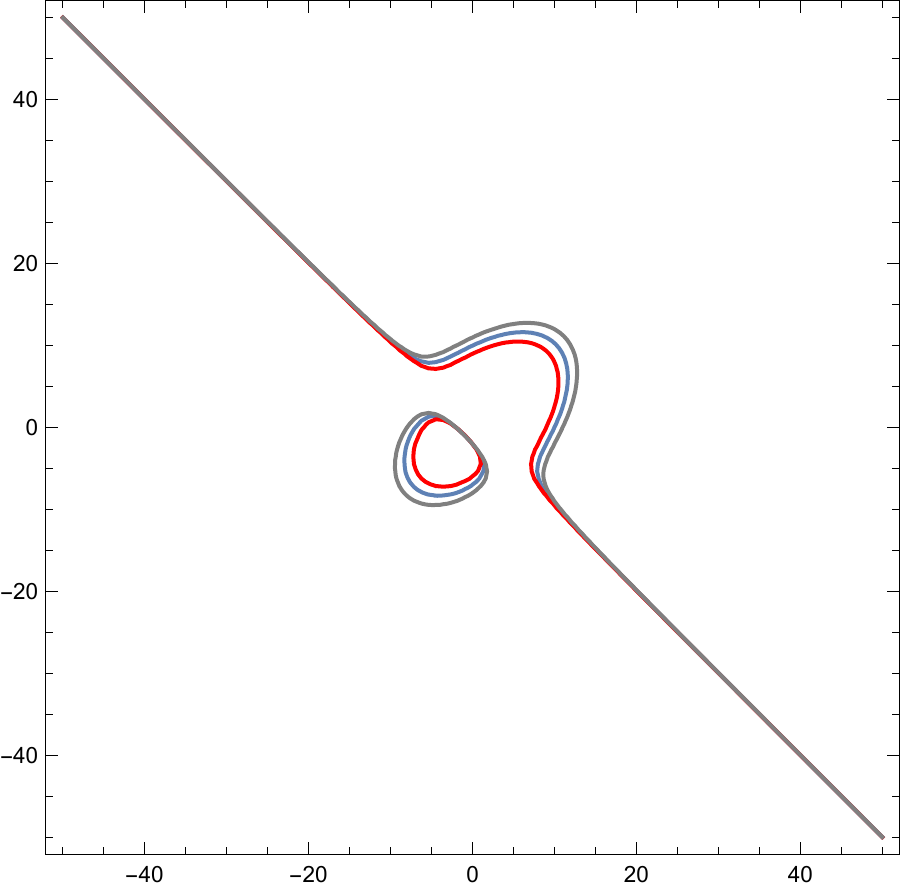}
    \caption {$\nu(w_n; a, b)=0$, where $n=7,8,9$.}
    \end{minipage}
\end {figure}
\par\noindent
Note that for all $n$ we have $\nu(w_n; -1, -1)= -6$ and every finite oval is tangent to the line
$$
a+b=-2
$$
at $(0, -2)$ and $(-2, 0)$. The infinite curve is asymptotic to the line
$$
a+b=0.
$$
\vskip0.5cm
\par

\end{document}